\documentclass[12pt]{amsart}%
\usepackage{amsmath}
\usepackage{amsfonts}
\usepackage{amssymb}
\usepackage{graphicx}
\usepackage{comment}
\usepackage[margin=1.5 in]{geometry}
\usepackage{setspace}%
\providecommand{\U}[1]{\protect\rule{.1in}{.1in}}
\providecommand{\U}[1]{\protect\rule{.1in}{.1in}}
\providecommand{\U}[1]{\protect\rule{.1in}{.1in}}
\newtheorem{theorem}{Theorem}

\newtheorem{corollary}[theorem]{Corollary}

\newtheorem{lemma}[theorem]{Lemma}

\allowdisplaybreaks[1]

\theoremstyle{definition}
\numberwithin{equation}{section}

\newcommand{\resumename}{R\'esum\'e}

\begin{document}
\date{\today}
\title[Haar Property]{Dihedral Group Frames with the Haar Property}
\author[V. Oussa]{Vignon Oussa, Brian Sheehan}
\address{Dept. of Mathematics\\
Bridgewater State University\\
Bridgewater, MA 02324 U.S.A.\\
 }
\email{bsheehan@student.bridgew.edu}
\maketitle

\begin{abstract}
We consider a unitary representation of the Dihedral group $D_{2n}%
=\mathbb{Z}_{n}\rtimes\mathbb{Z}_{2}$ obtained by inducing the trivial
character from the co-normal subgroup $\left\{0\right\}\rtimes\mathbb{Z}_{2}.$ This
representation is naturally realized as acting on the vector space
$\mathbb{C}^{n}.$ We prove that the orbit of almost every vector in
$\mathbb{C}^{n}$ with respect to the Lebesgue measure has the Haar property
(every subset of cardinality $n$ of the orbit is a basis for $\mathbb{C}^{n}$)
if $n$ is an odd integer. Moreover, we provide explicit sufficient conditions
for vectors in $\mathbb{C}^{n}$ whose orbits have the Haar property. Finally,
we derive that the orbit of almost every vector in $\mathbb{C}^{n}$ under the
action of the representation has the Haar property if and only if $n$ is odd.
This completely settles a problem which was only partially answered in
\cite{Oussa}.

\end{abstract}

\section{Introduction}

Let $\mathfrak{F}$ be a set of $m\geq n$ vectors in an $n$-dimensional vector
space $\mathbb{C}^{n}$ over $\mathbb{C}.$ We say that $\mathfrak{F}$ has the
\textbf{Haar property }or is\textbf{ a full spark frame} if any subset of
$\mathfrak{F}$ of cardinality $n$ is a basis for $\mathbb{C}^{n}$. We recall
that a set $\mathfrak{F}$ of vectors contained in $\mathbb{C}^{n}$ is called a
\textbf{frame} for $\mathbb{C}^{n}$ \cite{Casazza} if $\mathfrak{F}$ is a
spanning set for the vector space $\mathbb{C}^{n}.$ Moreover, it is well-known
that a countable family of vectors $\mathfrak{F}=\left(  v_{k}\right)  _{k\in
I}$ is a frame for $\mathbb{C}^{n}$ if and only if there exist positive
constants $A,B>0$ such that
\[
A\left\Vert v\right\Vert _{\mathbb{C}^{n}}^{2}\leq\sum_{k\in I}\left\langle
v,v_{k}\right\rangle _{\mathbb{C}^{n}}\leq B\left\Vert v\right\Vert
_{\mathbb{C}^{n}}^{2}%
\]
for every vector $v\in\mathbb{C}^{n}.$ Let
\[
T:\mathbb{C}^{\mathrm{card}\left(  I\right)  }\rightarrow\mathbb{C}%
^{n},T\left(  c\right)  =\sum_{k\in I}c_{k}v_{k}%
\]
be the synthesis operator corresponding to $\mathfrak{F}$. The adjoint of $T$
also known as the analysis operator is defined as follows
\[
T^{\ast}:\mathbb{C}^{n}\rightarrow\mathbb{C}^{\mathrm{card}\left(  I\right)
},\text{ }T^{\ast}v=\left(  \left\langle v,v_{k}\right\rangle _{\mathbb{C}%
^{n}}\right)  _{k\in I}.
\]
The composition of the synthesis operator with its adjoint%
\[
S:\mathbb{C}^{n}\rightarrow\mathbb{C}^{n},\text{ }Sv=\sum_{k\in I}\left\langle
v,v_{k}\right\rangle _{\mathbb{C}^{n}}v_{k}%
\]
is called the frame operator. It is a well-known fact \cite{Casazza} that if
$\mathfrak{F}=\left(  v_{k}\right)  _{k\in I}$ is a frame then $S$ is a
self-adjoint invertible operator and every vector $v\in\mathbb{C}^{n}$ can be
represented as
\[
v=\sum_{k\in I}\left\langle v,v_{k}\right\rangle _{\mathbb{C}^{n}}S^{-1}%
v_{k}.
\]
In other words, any vector $v$ can be reconstructed from the sequence $\left(
\left\langle v,v_{k}\right\rangle _{\mathbb{C}^{n}}\right)  _{k\in I}.$ In
fact if $K$ is a proper subset of $I$ such that $\left(  v_{k}\right)  _{k\in
K}$ is also a frame then any vector $v$ can be reconstructed from the sequence
$\left(  \left\langle v,v_{k}\right\rangle _{\mathbb{C}^{n}}\right)  _{k\in
K}$. Thus, full spark frames are very flexible basis-like tools in the sense
that if $\left(  v_{k}\right)  _{k\in I}$ is a full spark frame, then for any
subset $J$ of $I$ of cardinality $n,$ every vector $v\in\mathbb{C}^{n}$ can be
uniquely represented as
\[
v=\sum_{k\in J}\left\langle v,v_{k}\right\rangle _{\mathbb{C}^{n}}S_{J}%
^{-1}v_{k}\text{ with }S_{J}v=\sum_{k\in J}\left\langle v,v_{k}\right\rangle
_{\mathbb{C}^{n}}v_{k}.
\]
We shall now present some examples of full spark frames induced from the
action of group representations. Let $\mathbb{H}_{n}<GL\left(
n,\mathbb{\mathbb{C}}\right)  $ be the Heisenberg group which is a finite
group generated by
\[
T=\left[
\begin{array}
[c]{ccccc}%
0 & 0 & \cdots & 0 & 1\\
1 & 0 & \cdots & 0 & 0\\
0 & 1 & \ddots & \vdots & \vdots\\
\vdots & \ddots & \ddots & 0 & 0\\
0 & \cdots & 0 & 1 & 0
\end{array}
\right]  \text{ and }M=\left[
\begin{array}
[c]{ccccc}%
1 &  &  &  & \\
& \exp\left(  \frac{2\pi i}{n}\right)  &  &  & \\
&  & \exp\left(  \frac{4\pi i}{n}\right)  &  & \\
&  &  & \ddots & \\
&  &  &  & \exp\left(  \frac{2\pi i\left(  n-1\right)  }{n}\right)
\end{array}
\right]  .
\]
It is shown in \cite{Romanos, Pfander} that for almost every vector
$v\in\mathbb{C}^{n}$ (with respect to Lebesgue measure) the collection
\[
\left\{  T^{\ell}M^{\kappa}v:\left(  \ell,\kappa\right)  \in\mathbb{Z}%
_{n}\times\mathbb{Z}_{n}\right\}
\]
is a full spark frame. Next, let $n$ be a natural number greater than two. Let
$D_{2n}$ be the Dihedral group generated by $T,D\in GL\left(
n,\mathbb{\mathbb{C}}\right)  $ where
\begin{equation}
T=\left[
\begin{array}
[c]{ccccc}%
0 & 0 & \cdots & 0 & 1\\
1 & 0 & \cdots & 0 & 0\\
0 & 1 & \ddots & \vdots & \vdots\\
\vdots & \ddots & \ddots & 0 & 0\\
0 & \cdots & 0 & 1 & 0
\end{array}
\right]  ,D=\left[
\begin{array}
[c]{ccccc}%
1 & 0 & \cdots & 0 & 0\\
0 & 0 & \cdots & 0 & 1\\
\vdots & \vdots &  &  & 0\\
0 & 0 & 1 &  & \vdots\\
0 & 1 & 0 & \cdots & 0
\end{array}
\right]  .
\end{equation}
It is proved in \cite{Oussa} that if $n$ is even then it is not possible to
find a full spark frame of the type $D_{2n}v.$ Moreover, if $n$ is prime then
there exists a Zariski open subset $\Omega$ of $\mathbb{C}^{n}$ such that
$D_{2n}v$ is a full spark frame for every vector $v$ in $\Omega.$ While this
result establishes the existence of vectors $v$ such that $D_{2n}v$ has the
Haar property when $n$ is prime, the work in \cite{Oussa} is not constructive.
Moreover, the case where $n$ is odd but not prime was left open in
\cite{Oussa}. The main objective of the present paper is to fill this gap by
settling the following. Let $n$ be an odd number which is not necessarily odd.

\begin{itemize}
\item Is there a vector $v\in\mathbb{\mathbb{C}}^{n}$ such that $D_{2n}v$ has
the Haar property?

\item Can we provide an explicit construction of a vector $v\in
\mathbb{\mathbb{C}}^{n}$ such that $D_{2n}v$ has the Haar property?
\end{itemize}

\subsection{Main results}

Here is a summary of our results. For a fixed natural number $n>2$, we define
$\digamma_{n}\subset\mathbb{C}^{n}$ such that
\begin{equation}
\digamma_{n}=\left\{
\begin{array}
[c]{c}%
w\in\mathbb{C}^{n}:w=\left[
\begin{array}
[c]{c}%
\sum_{k=0}^{n-1}\lambda^{k}\\
\sum_{k=0}^{n-1}\lambda^{k}e^{-\frac{2\pi ki}{n}}\\
\vdots\\
\sum_{k=0}^{n-1}\lambda^{k}e^{-\frac{2\pi k\left(  n-1\right)  i}{n}}%
\end{array}
\right] \\
\text{where }\lambda\text{ is a either transcendal}\\
\text{or an algebraic number over }\mathbb{\mathbb{%
\mathbb{Q}
}}\left(  \frac{2\pi i}{n}\right)  \text{ whose degree is at least }n^{2}-n+1
\end{array}
\right\}  . \label{Fn}%
\end{equation}

\begin{theorem}
\label{main} If $n$ is odd and $w\in\digamma_{n}$ then $D_{2n}w$ is a full
spark frame.
\end{theorem}

Suppose now that $G$ is a finite group of matrices of order $m$ acting on
$\mathbb{C}^{n}.$ Moreover, let us assume that the order of $G$ is greater
than or equal to $n.$ It is not hard to verify that if there exists one vector
$w$ in $\mathbb{C}^{n}$ such that the orbit of $w$ under the action of $G$ has
the Haar property then for almost every vector $v$ (with respect to Lebesgue
measure) in $\mathbb{C}^{n},$ the orbit of $v$ under the action of $G$ has the
Haar property. Indeed, assume that there exists a vector $w$ in $\mathbb{C}%
^{n}$ such that $Gw$ has the Haar property. Consider the matrix-valued
function
\begin{equation}
v\mapsto O\left(  v\right)  =\left[
\begin{array}
[c]{ccc}%
\left(  g_{1}v\right)  _{0} & \cdots & \left(  g_{1}v\right)  _{n-1}\\
\vdots & \ddots & \vdots\\
\left(  g_{m}v\right)  _{0} & \cdots & \left(  g_{m}v\right)  _{n-1}%
\end{array}
\right]  .
\end{equation}
By assumption, there exists an element $O\left(  w\right)  $ in the range of
the function above such that any sub-matrix of order $n$ of $O\left(
w\right)  $ has a non-zero determinant. Let $P_{n}$ be the set of all
sub-matrices of $O\left(  v\right)  $ of order $n.$ Next, let $p\left(
v\right)  $ be the least common multiple of all polynomials obtained by
computing the determinant of $X\left(  v\right)  $ where $X\left(  v\right)
\in P_{n}.$ That is,
\[
p\left(  v\right)  =\operatorname{lcm}\left(  \left\{  \det\left(  X\left(
v\right)  \right)  :X\left(  v\right)  \in P_{n}\right\}  \right)  .
\]
Clearly $w$ is not a zero for the polynomial $p.$ Next, let $W$ be an open set
around $p\left(  w\right)  \neq0$ such that $W$ does not contain zero. Since
$p$ is continuous, the inverse image of the open set $W$ under the map $p$ is
an open subset of $\mathbb{C}^{n}.$ So, there exists an open subset of
$\mathbb{C}^{n}$ containing $w$ which is disjoint from the zero set of $p.$
This implies that $p$ is not the zero polynomial. Since the zero set of $p$ is
a co-null set with respect to the Lebesgue measure on $\mathbb{C}^{n}$, it
immediately follows that for almost every vector $v$ in $\mathbb{C}^{n},$ the
set $Gv$ has the Haar property. From the preceding discussion, the following
is immediate.

\begin{corollary}
\label{cor} Let $n$ be an odd natural number greater than two. There exists a
Zariski open subset $\Omega\subset\mathbb{C}^{n}$ such that if $v\in\Omega$
then $D_{2n}v$ is a full spark frame.
\end{corollary}

\noindent Finally, a direct application of the results in \cite{Oussa} and
Theorem \ref{main} together with Corollary \ref{cor} gives the following.

\begin{corollary}
The $D_{2n}$-orbit of almost every vector (with respect to the Lebesgue
measure) in $\mathbb{C}^{n}$ is a full spark frame if and only if $n$ is odd.
\end{corollary}

The paper is organized as follows. In the second section, we fix notation and
we review a variety of relevant concepts. The third section contains
intermediate results leading to the proof of the main result which is given in
the last section.

\section{Preliminaries}

Let us start by fixing notation. Given a matrix $M,$ the transpose of $M$ is
denoted $M^{T}.$ The determinant of a matrix $M$ is denoted by $\mathrm{det}%
(M)$ or $|M|.$ The $k$-th row of $M$ is denoted $\operatorname{row}_{k}\left(
M\right)  $ and similarly, the $k$-th column of the matrix $M$ is denoted
$\operatorname{col}_{k}\left(  M\right)  .$ Let $G$ be a group acting on a set
$S.$ We denote this action multiplicatively. For any fixed element $s\in S,$
the $G$-orbit of $s$ is described as $Gs=\left\{  gs:g\in G\right\}  .$

\subsection{Fourier Analysis on $\mathbb{Z}_{n}$}

Let $\mathbb{Z}_{n}=\left\{  0,1,\cdots,n-1\right\}  .$ Let $l^{2}\left(
\mathbb{Z}_{n}\right)  =\left\{  f:\mathbb{Z}_{n}\rightarrow\mathbb{C}%
\right\}  $ be the vector space of all complex-valued functions on
$\mathbb{Z}_{n}$ endowed with the dot product
\[
\left\langle \phi,\psi\right\rangle =\sum_{x\in\mathbb{Z}_{n}}\phi\left(
x\right)  \overline{\psi\left(  x\right)  }\text{ for }\phi,\psi\in
l^{2}\left(  \mathbb{Z}_{n}\right)  .
\]
The norm of a given vector $\phi$ in $l^{2}\left(  \mathbb{Z}_{n}\right)  $ is
computed as follows $\left\Vert \phi\right\Vert _{l^{2}\left(  \mathbb{Z}%
_{n}\right)  }=\left\langle \phi,\phi\right\rangle ^{1/2}.$ We recall that the
discrete Fourier transform is a map $\mathcal{F}:l^{2}\left(  \mathbb{Z}%
_{n}\right)  \rightarrow l^{2}\left(  \mathbb{Z}_{n}\right)  $ defined by
\[
\left[  \mathcal{F}\phi\right]  \left(  \xi\right)  =\frac{1}{n^{1/2}}%
\sum_{k\in\mathbb{Z}_{n}}\phi\left(  k\right)  \exp\left(  \frac{2\pi ik\xi
}{n}\right)  ,\text{ for }\phi\in l^{2}\left(  \mathbb{Z}_{n}\right)  .
\]
The following facts are also well-known \cite{Terras}.

\begin{itemize}
\item The discrete Fourier transform is a bijective linear operator.

\item The Fourier inverse of a vector $\varphi$ is given by
\[
\left[  \mathcal{F}^{-1}\varphi\right]  \left(  k\right)  =\frac{1}{n^{1/2}%
}\sum_{\xi\in\mathbb{Z}_{n}}\varphi\left(  \xi\right)  \exp\left(  -\frac{2\pi
ik\xi}{n}\right)  .
\]

\item The Fourier transform is a unitary operator. More precisely, given
$\phi,\psi\in l^{2}\left(  \mathbb{Z}_{n}\right) ,$ we have $\left\langle
\phi,\psi\right\rangle =\left\langle \mathcal{F}\phi,\mathcal{F}%
\psi\right\rangle .$
\end{itemize}

The followig is proved in \cite{Evans}.

\begin{lemma}
\label{minor cyclic}Let $\mathbf{F}$ be the matrix representation of the
Fourier transform. If $n$ is prime then every minor of $\mathbf{F}$ is nonzero.
\end{lemma}

Identifying $l^{2}\left(  \mathbb{Z}_{n}\right)  $ with $\mathbb{C}^{n}$ via
the map $v\mapsto\left[
\begin{array}
[c]{ccc}%
v\left(  0\right)  & \cdots & v\left(  n-1\right)
\end{array}
\right]  ^{T},$ we may write
\[
\left(  Tv\right)  \left(  j\right)  =v\left(  \left(  j-1\right)
\operatorname{mod}n\right)  \text{ and }\left(  Dv\right)  \left(  j\right)
=v\left(  \left(  n-j\right)  \operatorname{mod}n\right)  .
\]
It is not hard to verify that for any $\xi\in\mathbb{Z}_{n},$ we have $\left[
\mathcal{F}Bv\right]  \left(  \xi\right)  =\left(  \mathcal{F}v\right)
\left(  \left(  n-\xi\right)  \operatorname{mod}n\right)  $ and $\left[
\mathcal{F}Av\right]  \left(  \xi\right)  =e^{\frac{2\pi i}{n}\xi}\left(
\mathcal{F}v\right)  \left(  \xi\right) .$ Next, let $\mathbf{F}$ be the
matrix representation of the Fourier transform with respect to the canonical
basis of $\mathbb{C}^{n}.$ The matrix $T$ is diagonalizable and its Jordan
canonical form is obtained by conjugating $T$ by the discrete Fourier matrix.
More precisely if $M=\mathbf{F}T\mathbf{F}^{-1}$ then the $k^{th}$ diagonal
entry of $M$ is given by $M_{kk}=e^{\frac{2\left(  k-1\right)  \pi i}{n}}.$
Finally, it is worth noting that $D$ commutes with the Fourier matrix
$\mathbf{F.}$

\subsection{Determinants}

(\cite{Howard}, Section $3.7$) Let $X$ be a square matrix of order $n.$ A
minor of $X$ is the determinant of any square sub-matrix $Y$ of $X.$ Let
$\left\vert Y\right\vert =\det Y$ be an $m$-rowed minor of $X.$ The
determinant of the sub-matrix obtained by deleting from $X$ the rows and
columns represented in $Y$ is called the complement of $\left\vert
Y\right\vert $ and is denoted $\left\vert Y^{c}\right\vert =\det\left(
Y^{c}\right)  .$ Let $\left\vert Y\right\vert $ be the $m$-rowed minor of $X$
in which rows $i_{1},\cdots,i_{m}$ and columns $j_{1},\cdots,j_{m}$ are
represented. Then the algebraic complement, or cofactor of $\left\vert
Y\right\vert $ is given by $\left(  -1\right)  ^{\sum_{k=1}^{m}i_{k}%
+\sum_{k=1}^{m}j_{k}}\left\vert Y^{c}\right\vert .$ According to
\textbf{Laplace's Expansion Theorem} (see $3.7.3,$ \cite{Howard}) a formula
for the determinant of $X$ can be obtained as follows. Let $T\left(
n,p\right)  $ be the set of all ordered sets of cardinality $p$ of the type
$s=\left\{  0\leq s_{1}<s_{2}<\cdots<s_{p}\leq n-1\right\}  .$ Given any two
elements $s$ and $t$ in $T\left(  n,p\right)  ,$ we define $X\left(
s,t\right)  $ to be the sub-matrix of $X$ of order $p$ such that $X\left(
s,t\right)  _{i,j}=X_{s_{i}+1,t_{j}+1}.$ According to Laplace's Expansion
Theorem, once we fix an ordered set $t$ in $T\left(  n,p\right)  ,$ the
determinant of $X$ can be computed via the following expansion%
\begin{equation}
\sum_{s\in T\left(  n,p\right)  }\left(  -1\right)  ^{\sum_{k=1}^{p}\left(
s_{k}+1\right)  +\sum_{k=1}^{p}\left(  t_{k+1}\right)  }\det\left(  X\left(
t,s\right)  \right)  \det\left(  X\left(  t,s\right)  ^{c}\right)  .
\end{equation}
The following well-known result will also be useful to us.

\begin{lemma}
(Page $127,$\cite{Howard}) Let
\[
V=\left[
\begin{array}
[c]{ccccc}%
1 & a_{1} & a_{1}^{2} & \cdots & a_{1}^{n-1}\\
1 & a_{2} & a_{2}^{2} & \cdots & a_{2}^{n-1}\\
\vdots & \vdots & \vdots & \cdots & \vdots\\
1 & a_{n} & a_{n}^{2} & \cdots & a_{n}^{n-1}%
\end{array}
\right]
\]
be a \textbf{Vandermonde matrix}. Then
\begin{equation}
\det\left(  V\right)  =%
{\displaystyle\prod\limits_{1\leq j<k\leq n}}
\left(  a_{k}-a_{j}\right)  \label{Vand}%
\end{equation}

\end{lemma}

\section{Intermediate Results}

In this section, we will prove several results which lead to the proof of
Theorem \ref{main}. Put
\begin{equation}
\xi=e^{\frac{2\pi i}{n}}. \label{root of unity}%
\end{equation}
Let $\iota\left(  k\right)  $ is be the inverse of $k$ modulo $n.$ That is
$\iota\left(  k\right)  =\left(  n-k\right)  \operatorname{mod}n.$ Let $M$ be
the linear operator obtained by conjugating $T$ by the Fourier transform. Let
$\Lambda$ be the group of linear operators generated by $M$ and $D.$ Fixing an
ordering for the elements of $\Lambda$, we may write $\Lambda=\left\{
\gamma_{1},\cdots,\gamma_{2n}\right\}  .$ We identify the Hilbert space
$l^{2}\left(  \mathbb{Z}_{n}\right)  $ in the usual way with the Cartesian
product of $n$ copies of the complex plane. For any $z\in\mathbb{C}^{n},$ we
define the map $z\mapsto A\left(  z\right)  $ where $A\left(  z\right)  $ is
the $2n$ by $n$ matrix-valued function whose $k^{th}$ row corresponds to the
vector $\gamma_{k}z$ which is an element of the $\Lambda$-orbit of $z.$ From
our definition, $A$ is regarded as a matrix-valued function taking the vector
$z$ to the matrix
\[
A\left(  z\right)  =\left[
\begin{array}
[c]{ccc}%
\left(  \gamma_{1}z\right)  _{0} & \cdots & \left(  \gamma_{1}z\right)
_{n-1}\\
\vdots & \ddots & \vdots\\
\left(  \gamma_{2n}z\right)  _{0} & \cdots & \left(  \gamma_{2n}z\right)
_{n-1}%
\end{array}
\right]  .
\]
Let $S\left(  z\right)  $ be a fixed submatrix of $A\left(  z\right)  $ of
order $n.$ Fixing an ordering for the rows of $S\left(  z\right)  $, we may
assume that the following holds true.

\begin{description}
\item[Case $\left(  a\right)  $] Either
\begin{equation}
S\left(  z\right)  =\left[
\begin{array}
[c]{cccc}%
z_{0} & z_{1} & \cdots & z_{n-1}\\
z_{0} & \xi z_{1} & \cdots & \xi^{\left(  n-1\right)  }z_{n-1}\\
\vdots & \ddots & \ddots & \vdots\\
z_{0} & \xi^{\left(  n-1\right)  }z_{1} & \cdots & \xi^{\left(  n-1\right)
\left(  n-1\right)  }z_{n-1}%
\end{array}
\right]  \text{ }%
\end{equation}
or
\begin{equation}
S\left(  z\right)  =\left[
\begin{array}
[c]{cccc}%
z_{0} & z_{\iota\left(  1\right)  } & \cdots & z_{i\left(  n-1\right)  }\\
z_{0} & \xi z_{\iota\left(  1\right)  } & \cdots & \xi^{\left(  n-1\right)
}z_{i\left(  n-1\right)  }\\
\vdots & \ddots & \ddots & \vdots\\
z_{0} & \xi^{\left(  n-1\right)  }z_{\iota\left(  1\right)  } & \cdots &
\xi^{\left(  n-1\right)  \left(  n-1\right)  }z_{i\left(  n-1\right)  }%
\end{array}
\right]
\end{equation}

\item[Case $\left(  b\right)  $] or $S\left(  z\right)  $ is equal to%
\begin{equation}
\left[
\begin{array}
[c]{cccccccc}%
z_{0} & \xi^{1l_{0}}z_{1} & \cdots & \xi^{\left(  m-1\right)  l_{0}}z_{m-1} &
\xi^{ml_{0}}z_{m} & \xi^{\left(  m+1\right)  l_{0}}z_{m+1} & \cdots &
\xi^{\left(  n-1\right)  l_{0}}z_{n-1}\\
z_{0} & \xi^{1l_{1}}z_{1} & \cdots & \xi^{\left(  m-1\right)  l_{1}}z_{m-1} &
\xi^{ml_{1}}z_{m} & \xi^{\left(  m+1\right)  l_{1}}z_{m+1} & \cdots &
\xi^{\left(  n-1\right)  l_{1}}z_{n-1}\\
\vdots & \vdots & \ddots & \vdots & \vdots & \vdots & \ddots & \vdots\\
z_{0} & \xi^{1l_{m-1}}z_{1} & \cdots & \xi^{\left(  m-1\right)  l_{m-1}%
}z_{m-1} & \xi^{ml_{m-1}}z_{m} & \xi^{\left(  m+1\right)  l_{m-1}}z_{m+1} &
\cdots & \xi^{\left(  n-1\right)  l_{m-1}}z_{n-1}\\
z_{0} & \xi^{1j_{0}}z_{\iota\left(  1\right)  } & \cdots & \xi^{\left(
m-1\right)  j_{0}}z_{\iota\left(  m-1\right)  } & \xi^{mj_{0}}z_{\iota\left(
m\right)  } & \xi^{\left(  m+1\right)  j_{0}}z_{\iota\left(  m+1\right)  } &
\cdots & \xi^{\left(  n-1\right)  j_{0}}z_{i\left(  n-1\right)  }\\
z_{0} & \xi^{1j_{1}}z_{\iota\left(  1\right)  } & \cdots & \xi^{\left(
m-1\right)  j_{1}}z_{\iota\left(  m-1\right)  } & \xi^{mj_{1}}z_{\iota\left(
m\right)  } & \xi^{\left(  m+1\right)  j_{1}}z_{\iota\left(  m+1\right)  } &
\cdots & \xi^{\left(  n-1\right)  j_{1}}z_{i\left(  n-1\right)  }\\
\vdots & \vdots & \ddots & \vdots & \vdots & \vdots & \ddots & \vdots\\
z_{0} & \xi^{1j_{s-1}}z_{\iota\left(  1\right)  } & \cdots & \xi^{\left(
m-1\right)  j_{s-1}}z_{\iota\left(  m-1\right)  } & \xi^{mj_{s-1}}%
z_{\iota\left(  m\right)  } & \xi^{\left(  m+1\right)  j_{s-1}}z_{\iota\left(
m+1\right)  } & \cdots & \xi^{\left(  n-1\right)  j_{s-1}}z_{i\left(
n-1\right)  }%
\end{array}
\right]
\end{equation}
for some integers $l_{0},\cdots,l_{m-1},j_{0},j_{1},\cdots,j_{s-1}$ such that
$m+s=n,$ if $k\neq k^{\prime}$ then $l_{k}\neq l_{k^{\prime}}$ and if $k\neq
k^{\prime}$ then $j_{k}\neq j_{k^{\prime}}.$
\end{description}

Let $\mathbf{q}$ be the non-trivial polynomial defined by
\begin{equation}
\mathbf{q}\left(  t\right)  ={\prod\limits_{k=1}^{n-1}}t^{k}.
\end{equation}

\subsection{Case $\left(  a\right)  $}

Let us consider the complex curve
\begin{equation}%
\mathbb{C}
\backepsilon t\mapsto\mathfrak{c}\left(  t\right)  =\left[
\begin{array}
[c]{ccccc}%
1 & t & t^{2} & \cdots & t^{n-1}%
\end{array}
\right]  ^{T}\in\mathbb{C}^{n}.
\end{equation}
The image of this curve under the function $S$ is simply computed by replacing
each variable $z_{k}$ by $t^{k}.$ If we assume that $S\left(  z\right)  $ is
equal to the first matrix described in Case $\left(  a\right)  $ then
\begin{equation}
\det\left(  S\left(  1,t,t^{2},\cdots,t^{n-1}\right)  \right)  =\mathbf{q}%
\left(  t\right)  \left\vert
\begin{array}
[c]{cccc}%
1 & 1 & \cdots & 1\\
1 & \xi & \cdots & \xi^{\left(  n-1\right)  }\\
1 & \xi^{2} & \cdots & \xi^{2\left(  n-1\right)  }\\
\vdots & \ddots & \ddots & \vdots\\
1 & \xi^{\left(  n-1\right)  } & \cdots & \xi^{\left(  n-1\right)  \left(
n-1\right)  }%
\end{array}
\right\vert =\left(  {\prod\limits_{0\leq j<k\leq n-1}}\left(  \xi^{k}-\xi
^{j}\right)  \right)  \times{\prod\limits_{k=1}^{n-1}}t^{k}.
\end{equation}
In a similar fashion, it is easy to verify that if $S\left(  z\right)  $ is
equal to the second matrix described in Case $\left(  a\right)  $ then
\[
\det\left(  S\left(  1,t,t^{2},\cdots,t^{n-1}\right)  \right)  =\left(
{\prod\limits_{0\leq j<k\leq n-1}}\left(  \xi^{k}-\xi^{j}\right)  \right)
\times\mathbf{q}\left(  t\right)  .
\]
Thus, the determinants of the matrices described in Case (a) are non-zero polynomials.

\subsection{Case $\left(  b\right)  $}

Define matrices $A_{1}\left(  t\right)  ,A_{2}\left(  t\right)  $ as
\[
\left[
\begin{array}
[c]{ccc}%
\xi^{0l_{0}}t & \cdots & \xi^{\left(  m-1\right)  l_{0}}t^{m-1}\\
\xi^{0l_{1}}t & \cdots & \xi^{\left(  m-1\right)  l_{1}}t^{m-1}\\
\vdots & \ddots & \vdots\\
\xi^{0l_{m-1}}t & \cdots & \xi^{\left(  m-1\right)  l_{m-1}}t^{m-1}%
\end{array}
\right]  ,\left[
\begin{array}
[c]{cccc}%
\xi^{ml_{0}}t^{m} & \xi^{\left(  m+1\right)  l_{0}}t^{m+1} & \cdots &
\xi^{\left(  n-1\right)  l_{0}}t^{n-1}\\
\xi^{ml_{1}}t^{m} & \xi^{\left(  m+1\right)  l_{1}}t^{m+1} & \cdots &
\xi^{\left(  n-1\right)  l_{1}}t^{n-1}\\
\vdots & \vdots & \ddots & \vdots\\
\xi^{ml_{m-1}}t^{m} & \xi^{\left(  m+1\right)  l_{m-1}}t^{m+1} & \cdots &
\xi^{\left(  n-1\right)  l_{m-1}}t^{n-1}%
\end{array}
\right]
\]
respectively. Additionally, we define $A_{3}\left(  t\right)  ,$ and
$A_{4}\left(  t\right)  $ as
\[
\left[
\begin{array}
[c]{ccc}%
\xi^{0\times j_{0}}t^{n-1} & \cdots & \xi^{\left(  m-1\right)  j_{0}%
}t^{n-\left(  m-1\right)  }\\
\xi^{0\times j_{1}}t^{n-1} & \cdots & \xi^{\left(  m-1\right)  j_{1}%
}t^{n-\left(  m-1\right)  }\\
\vdots & \ddots & \vdots\\
\xi^{0\times j_{s-1}}t^{n-1} & \cdots & \xi^{\left(  m-1\right)  j_{s-1}%
}t^{n-\left(  m-1\right)  }%
\end{array}
\right]  \text{ and }\left[
\begin{array}
[c]{cccc}%
\xi^{mj_{0}}t^{n-m} & \xi^{\left(  m+1\right)  j_{0}}t^{n-m-1} & \cdots &
\xi^{\left(  n-1\right)  j_{0}}t\\
\xi^{mj_{1}}t^{n-m} & \xi^{\left(  m+1\right)  j_{1}}t^{n-m-1} & \cdots &
\xi^{\left(  n-1\right)  j_{1}}t\\
\vdots & \vdots & \ddots & \vdots\\
\xi^{mj_{s-1}}t^{n-m} & \xi^{m+1j_{s-1}}t^{n-m-1} & \cdots & \xi^{\left(
n-1\right)  j_{s-1}}t
\end{array}
\right]
\]
respectively. Note that $A_{1}\left(  t\right)  $ is a square matrix of order
$m,$ $A_{4}\left(  t\right)  $ is a square matrix of order $n-m=s,$
$A_{2}\left(  t\right)  $ and $A_{3}\left(  t\right)  $ are matrices of order
$m\times s$ and $s\times m$ respectively. If we assume that $S\left(
z\right)  $ is as described in Case $\left(  b\right)  $ then $S\left(
\mathfrak{c}\left(  t\right)  \right)  $ is given in block form as follows
\begin{equation}
\left[
\begin{array}
[c]{cc}%
A_{1}\left(  t\right)   & A_{2}\left(  t\right)  \\
A_{3}\left(  t\right)   & A_{4}\left(  t\right)
\end{array}
\right]  .\label{sub}%
\end{equation}

\begin{lemma}
The following holds true.

\begin{enumerate}
\item $\det A_{1}\left(  t\right)  ={\prod\limits_{0\leq k<j\leq m-1}}\left(
\xi^{l_{j}}-\xi^{l_{k}}\right)  \left(  t\cdots t^{m-1}\right)  .$

\item $\det A_{4}\left(  t\right)  =\left(  \xi^{mj_{0}}\xi^{mj_{1}}\cdots
\xi^{mj_{s-1}}\right)  {\prod\limits_{0\leq k<j\leq n-1-m}}\left(
\xi^{j_{\ell}}-\xi^{j_{k}}\right)  \left(  t^{n-m}t^{n-m-1}\cdots
t^{2}t\right)  .$
\end{enumerate}
\end{lemma}

\begin{proof}
The multi-linearity of the determinant function together with Formula
(\ref{Vand}) give the desired result.
\end{proof}

\noindent Let $\mathbf{p}_{m}=\left\{  p_{1}<p_{2}<\cdots<p_{m-1}%
<p_{m}\right\}  $ be an ordered subset of $\mathbb{Z}_{n}$ of cardinality $m$
and set
\begin{equation}
\mathbf{i}_{m}=\left\{  0<1<\cdots<m-1\right\}  .\label{i(m)}%
\end{equation}
Let $S\left(  \mathfrak{c}\left(  t\right)  \right)  $ be the matrix given in
(\ref{sub}). To avoid cluster of notation, from now on, we shall write
$S\left(  \mathfrak{c}\left(  t\right)  \right)  =\mathbf{S}_{t}\mathbf{.}$
Put
\begin{equation}
C=\left(  \xi^{mj_{0}}\xi^{mj_{1}}\cdots\xi^{mj_{s-1}}\right)  {\prod
\limits_{0\leq k<\ell\leq n-1-m}}\left(  \xi^{j_{\ell}}-\xi^{j_{k}}\right)
{\prod\limits_{0\leq k<j\leq m-1}}\left(  \xi^{l_{j}}-\xi^{l_{k}}\right)  \in%
\mathbb{Q}
\left(  \xi\right)  .\label{C}%
\end{equation}

\noindent For $n,m\in\mathbb{N}$ such that $n>2$ and $m\leq n,$ define
\begin{align*}
\kappa_{\left(  n,m\right)  }  & ={\sum\limits_{k=1}^{m-1}}\left(  k\right)
+{\sum\limits_{k=1}^{n-m}}\left(  k\right)  =\frac{1}{2}m\left(  m-1\right)
-\frac{1}{2}\left(  m-n\right)  \left(  n-m+1\right)  \\
& =\frac{2m^{2}-2mn-2m+n^{2}+n}{2}\in%
\mathbb{N}
.
\end{align*}

\begin{lemma}
\label{im}The product of the determinants of $\mathbf{S}_{t}\left(
\mathbf{i}_{m},\mathbf{i}_{m}\right)  $ and $\mathbf{S}_{t}\left(
\mathbf{i}_{m},\mathbf{i}_{m}\right)  ^{c}$ is equal to the non-trivial
polynomial $Ct^{\kappa_{\left(  n,m\right)  }}.$
\end{lemma}

\begin{proof}
We check that the determinant of the matrix $\mathbf{S}_{t}\left(
\mathbf{i}_{m},\mathbf{i}_{m}\right)  $ is a monomial in the variable $t$
which is given by
\begin{equation}
\left(  t\cdots t^{m-1}\right)  \cdot{\prod\limits_{0\leq k<j\leq m-1}}\left(
\xi^{l_{j}}-\xi^{l_{k}}\right)  .\label{product 1}%
\end{equation}
Moreover, the determinant of $\left[  \mathbf{S}_{t}\left(  \mathbf{i}%
_{m},\mathbf{i}_{m}\right)  \right]  ^{c}$ is given by%
\begin{equation}
\left(  \xi^{mj_{0}}\xi^{mj_{1}}\cdots\xi^{mj_{s-1}}\right)  \cdot
{\prod\limits_{0\leq k<\ell\leq n-1-m}}\left(  \xi^{j_{\ell}}-\xi^{j_{k}%
}\right)  \cdot t^{n-m}t^{n-m-1}\cdots t^{2}t.\label{product 2}%
\end{equation}
In summary, taking the product of (\ref{product 1}) and (\ref{product 2}) we
obtain $Ct^{\kappa_{\left(  n,m\right)  }}.$ Moreover, since $\xi^{j_{\ell}%
}-\xi^{j_{k}}\neq0$ and $\xi^{l_{j}}-\xi^{l_{k}}\neq0$ for $0\leq k<j\leq
n-1-m$ and $0\leq k<j\leq m-1$, it must be case that the monomial
$Ct^{\kappa_{\left(  n,m\right)  }}$ is a non-trivial polynomial.
\end{proof}

\noindent Observe that $\mathbf{p}_{m}=\left(  \mathbf{i}_{m}\cap
\mathbf{p}_{m}\right)  \cup\left(  \mathbf{p}_{m}-\mathbf{i}_{m}\right)  $ and
$\mathbf{p}_{m}-\mathbf{i}_{m}$ and $\mathbf{i}_{m}-\mathbf{p}_{m}$ are
subsets of $\mathbb{Z}_{n}$ of the same cardinality. Furthermore if
$\mathbf{p}_{m}$ and $\mathbf{i}_{m}$ are not equal as sets, we shall write
$\mathbf{i}_{m}=\left(  \mathbf{i}_{m}\cap\mathbf{p}_{m}\right)  \cup\left(
\mathbf{i}_{m}-\mathbf{p}_{m}\right)  $ and
\begin{align*}
\mathbf{i}_{m}-\mathbf{p}_{m}  &  =\left\{  \mathbf{i}_{1}<\cdots
<\mathbf{i}_{u\left(  \mathbf{p}_{m}\right)  }\right\}  \text{, }\\
\mathbf{p}_{m}-\mathbf{i}_{m}  &  =\left\{  \mathbf{i}_{1}+\ell_{1}\left(
\mathbf{p}_{m}\right)  <\cdots<\mathbf{i}_{u\left(  \mathbf{p}_{m}\right)
}+\ell_{u\left(  \mathbf{p}_{m}\right)  }\left(  \mathbf{p}_{m}\right)
\right\}
\end{align*}
such that each $\ell_{k}\left(  \mathbf{p}_{m}\right)  $ is a positive number
belonging to the set $\mathbb{Z}_{n}$, and $u\left(  \mathbf{p}_{m}\right)
=\left\vert \mathbf{i}_{m}-\mathbf{p}_{m}\right\vert >0.$ We recall that
$\mathbf{S}_{t}\left(  \mathbf{i}_{m},\mathbf{p}_{m}\right)  $ is the matrix
obtained by selecting all entries of the matrix $\mathbf{S}_{t}$ of the type
$\left[  \mathbf{S}_{t}\right]  _{k,k^{\prime}}$ where $k-1$ is in
$\mathbf{i}_{m}$ and $k^{\prime}-1$ is an element of $\mathbf{p}_{m}.$

\begin{lemma}
If $n$ is odd then
\[
\det\left(  \mathbf{S}_{t}\left(  \mathbf{i}_{m},\mathbf{p}_{m}\right)
\right)  \det\left(  \mathbf{S}_{t}\left(  \mathbf{i}_{m},\mathbf{p}%
_{m}\right)  ^{c}\right)  =Ct^{\kappa_{\left(  n,m\right)  }}%
\Longleftrightarrow\mathbf{i}_{m}=\mathbf{p}_{m}.
\]

\end{lemma}

\begin{proof}
According to Lemma \ref{im}, if $\mathbf{i}_{m}=\mathbf{p}_{m}$ then
\[
\det\left(  \mathbf{S}_{t}\left(  \mathbf{i}_{m},\mathbf{p}_{m}\right)
\right)  \det\left(  \mathbf{S}_{t}\left(  \mathbf{i}_{m},\mathbf{p}%
_{m}\right)  ^{c}\right)  =Ct^{\kappa_{\left(  n,m\right)  }}.
\]
For the other way around, let us suppose that $\mathbf{i}_{m}\neq
\mathbf{p}_{m}.$ Then $u\left(  \mathbf{p}_{m}\right)  =\mathrm{card}\left(
\mathbf{i}_{m}-\mathbf{p}_{m}\right)  >0.$ First, we observe that, there
exists integer $\delta$ and a complex number $q$ such that
\[
\det\left(  \mathbf{S}_{t}\left(  \mathbf{i}_{m},\mathbf{p}_{m}\right)
\right)  \det\left(  \mathbf{S}_{t}\left(  \mathbf{i}_{m},\mathbf{p}%
_{m}\right)  ^{c}\right)  =qt^{\delta}.
\]
Secondly, the monomial $t^{\delta}$ can be computed by multiplying
\[
t^{\kappa_{\left(  n,m\right)  }}=\left(  t\cdots t^{m-1}\right)  \left(
t^{n-m}t^{n-\left(  m+1\right)  }\cdots t^{2}t\right)
\]
by a suitable rational function of the variable $t$. Indeed, since the set
$\mathbf{p}_{m}$ is obtained by taking the union of $\mathbf{i}_{m}%
\cap\mathbf{p}_{m}$ and $\mathbf{p}_{m}-$\textbf{$\mathbf{i}_{m}$}$\mathbf{,}$
the above-mentioned monomial $t^{\delta}$ is the product of $\left(  t\cdots
t^{m-1}\right)  \left(  t^{n-m}t^{n-\left(  m+1\right)  }\cdots t^{2}t\right)
$ and $r\left(  t\right)  $ where
\[
r\left(  t\right)  =\dfrac{t^{\sum_{k\in\iota\left(  \mathbf{i}_{m}%
-\mathbf{p}_{m}\right)  \cup\mathbf{p}_{m}-\mathbf{i}_{m}}\left(  k\right)  }%
}{t^{\sum_{k\in\left(  \mathbf{i}_{m}-\mathbf{p}_{m}\right)  \cup\iota\left(
\mathbf{p}_{m}-\mathbf{i}_{m}\right)  }\left(  k\right)  }}.
\]
Next, it is clear that $r\left(  t\right)  =t^{\sigma}$ for some integer
$\sigma,$ and $\sigma$ is computed as follows. If zero is an element of
$\mathbf{i}_{m}\mathbf{-\mathbf{p}_{m}}$ then
\begin{align*}
\sigma &  =%
{\displaystyle\sum\limits_{k=1}^{u\left(  \mathbf{p}_{m}\right)  }}
\iota\left(  \mathbf{i}_{k}\right)  +\left(  \mathbf{i}_{k}+\ell_{k}\left(
\mathbf{p}_{m}\right)  \right)  -%
{\displaystyle\sum\limits_{k=1}^{u\left(  \mathbf{p}_{m}\right)  }}
\mathbf{i}_{k}+\iota\left(  \mathbf{i}_{k}+\ell_{k}\left(  \mathbf{p}%
_{m}\right)  \right)  \\
&  =\iota\left(  \mathbf{i}_{1}\right)  +\left(  \mathbf{i}_{1}+\ell
_{1}\left(  \mathbf{p}_{m}\right)  \right)  +%
{\displaystyle\sum\limits_{k=2}^{u\left(  \mathbf{p}_{m}\right)  }}
\iota\left(  \mathbf{i}_{k}\right)  +\left(  \mathbf{i}_{k}+\ell_{k}\left(
\mathbf{p}_{m}\right)  \right)  -%
{\displaystyle\sum\limits_{k=1}^{u\left(  \mathbf{p}_{m}\right)  }}
\mathbf{i}_{k}+\iota\left(  \mathbf{i}_{k}+\ell_{k}\left(  \mathbf{p}%
_{m}\right)  \right)  \\
&  =\iota\left(  \mathbf{i}_{1}\right)  +\left(  \mathbf{i}_{1}+\ell
_{1}\left(  \mathbf{p}_{m}\right)  \right)  -\mathbf{i}_{1}-\iota\left(
\mathbf{i}_{1}+\ell_{1}\left(  \mathbf{p}_{m}\right)  \right)  \\
&  +%
{\displaystyle\sum\limits_{k=2}^{u\left(  \mathbf{p}_{m}\right)  }}
\iota\left(  \mathbf{i}_{k}\right)  +\left(  \mathbf{i}_{k}+\ell_{k}\left(
\mathbf{p}_{m}\right)  \right)  -%
{\displaystyle\sum\limits_{k=2}^{u\left(  \mathbf{p}_{m}\right)  }}
\mathbf{i}_{k}+\iota\left(  \mathbf{i}_{k}+\ell_{k}\left(  \mathbf{p}%
_{m}\right)  \right)  \\
&  =\left(  \ast\right)
\end{align*}
Next,
\begin{align*}
\left(  \ast\right)   &  =\ell_{1}\left(  \mathbf{p}_{m}\right)  -\iota\left(
\ell_{1}\left(  \mathbf{p}_{m}\right)  \right)  +%
{\displaystyle\sum\limits_{k=2}^{u\left(  \mathbf{p}_{m}\right)  }}
\left(  \iota\left(  \mathbf{i}_{k}\right)  +\left(  \mathbf{i}_{k}+\ell
_{k}\left(  \mathbf{p}_{m}\right)  \right)  \right)  -\left(  \mathbf{i}%
_{k}+\iota\left(  \mathbf{i}_{k}+\ell_{k}\left(  \mathbf{p}_{m}\right)
\right)  \right)  \\
&  =\ell_{1}\left(  \mathbf{p}_{m}\right)  -\left(  n-\ell_{1}\left(
\mathbf{p}_{m}\right)  \right)  +%
{\displaystyle\sum\limits_{k=2}^{u\left(  \mathbf{p}_{m}\right)  }}
\left(  \iota\left(  \mathbf{i}_{k}\right)  +\left(  \mathbf{i}_{k}+\ell
_{k}\left(  \mathbf{p}_{m}\right)  \right)  \right)  -\left(  \mathbf{i}%
_{k}+\iota\left(  \mathbf{i}_{k}+\ell_{k}\left(  \mathbf{p}_{m}\right)
\right)  \right)  \\
&  =\ell_{1}\left(  \mathbf{p}_{m}\right)  -\left(  n-\ell_{1}\left(
\mathbf{p}_{m}\right)  \right)  +%
{\displaystyle\sum\limits_{k=2}^{u\left(  \mathbf{p}_{m}\right)  }}
2\ell_{k}\left(  \mathbf{p}_{m}\right)  \\
&  =-n+2\ell_{1}\left(  \mathbf{p}_{m}\right)  +%
{\displaystyle\sum\limits_{k=2}^{u\left(  \mathbf{p}_{m}\right)  }}
2\ell_{k}\left(  \mathbf{p}_{m}\right)  \\
&  =\left(
{\displaystyle\sum\limits_{k=1}^{u\left(  \mathbf{p}_{m}\right)  }}
2\ell_{k}\left(  \mathbf{p}_{m}\right)  \right)  -n.
\end{align*}
Since $n$ is odd, the integer $\sigma$ cannot be equal to zero. On the other
hand, if zero is not an element of $\mathbf{i-p}$ then
\begin{align*}
\sigma &  =%
{\displaystyle\sum\limits_{k=1}^{u\left(  \mathbf{p}_{m}\right)  }}
n-\mathbf{i}_{k}+\left(  \mathbf{i}_{k}+\ell_{k}\left(  \mathbf{p}_{m}\right)
\right)  -%
{\displaystyle\sum\limits_{k=1}^{u\left(  \mathbf{p}_{m}\right)  }}
\mathbf{i}_{k}+n-\left(  \mathbf{i}_{k}+\ell_{k}\left(  \mathbf{p}_{m}\right)
\right)  \\
&  =%
{\displaystyle\sum\limits_{k=1}^{u\left(  \mathbf{p}_{m}\right)  }}
\left(  n-\mathbf{i}_{k}+\left(  \mathbf{i}_{k}+\ell_{k}\left(  \mathbf{p}%
_{m}\right)  \right)  \right)  -\left(  \mathbf{i}_{k}+n-\left(
\mathbf{i}_{k}+\ell_{k}\left(  \mathbf{p}_{m}\right)  \right)  \right)  =%
{\displaystyle\sum\limits_{k=1}^{u\left(  \mathbf{p}_{m}\right)  }}
2\ell_{k}\left(  \mathbf{p}_{m}\right)  .
\end{align*}
In this case, since each $\ell_{k}\left(  \mathbf{p}_{m}\right)  $ is
positive, $r\left(  t\right)  $ must be a non-trivial homogeneous polynomial
in the variable $t.$ It follows that
\begin{align*}
\det\left(  \mathbf{S}_{t}\left(  \mathbf{i}_{m},\mathbf{p}_{m}\right)
\right)  \det\left(  \mathbf{S}_{t}\left(  \mathbf{i}_{m},\mathbf{p}%
_{m}\right)  ^{c}\right)   &  =q\left(  t\cdots t^{m-1}\right)  \left(
t^{n-m}t^{n-\left(  m+1\right)  }\cdots t^{2}t\right)  r\left(  t\right)  \\
&  =q\left(  t\cdots t^{m-1}\right)  \left(  t^{n-m}t^{n-\left(  m+1\right)
}\cdots t^{2}t\right)  t^{\delta}%
\end{align*}
and $\delta$ is a not equal to zero. Thus,
\[
\det\left(  \mathbf{S}_{t}\left(  \mathbf{i}_{m},\mathbf{p}_{m}\right)
\right)  \det\left(  \mathbf{S}_{t}\left(  \mathbf{i}_{m},\mathbf{p}%
_{m}\right)  ^{c}\right)  \neq Ct^{\frac{2m^{2}-2mn-2m+n^{2}+n}{2}}.
\]

\end{proof}

\begin{lemma}
The following holds true.

\begin{enumerate}
\item Every submatrix of order $n$ of $A\left(  \mathfrak{c}\left(  t\right)
\right) $ is a non-vanishing element of the ring $\mathbb{Q}\left(
\xi\right)  \left[  t\right] .$

\item Every submatrix of order $n$ of $A\left(  \mathfrak{c}\left(  t\right)
\right)  $ is a non-vanishing element of the ring $\mathbb{Q}\left(
\xi\right)  \left[  t\right]  $ of degree less than or equal to $n\left(
n-1\right)  .$
\end{enumerate}
\end{lemma}

\begin{proof}
For the first part, let $S\left(  \mathfrak{c}\left(  t\right)  \right)  $ be
a submatrix of $A\left(  \mathfrak{c}\left(  t\right)  \right)  $ of order
$n$. Assume that $S\left(  \mathfrak{c}\left(  t\right)  \right)  $ is given
as shown in (\ref{sub}). Let
\[
T\left(  n,m\right)  =\left\{  s=\left\{  0\leq s_{1}<s_{2}<\cdots<s_{m}\leq
n-1\right\}  \subseteq\mathbb{Z}_{n}\right\}  .
\]
We fix $t\in T\left(  n,m\right)  .$ Put $\mathbf{S}_{t}=S\left(
\mathfrak{c}\left(  t\right)  \right)  .$ According to Laplace's Expansion
Theorem,
\begin{equation}
\det\left(  \mathbf{S}_{t}\right)  =\det\left(  \mathbf{S}_{t}\left(
t,s\right)  \right)  =\sum_{s\in T\left(  n,p\right)  }\left(  -1\right)
^{\left\vert s\right\vert +\left\vert t\right\vert }\det\left(  \mathbf{S}%
_{t}\left(  t,s\right)  \right)  \det\left(  \mathbf{S}_{t}\left(  t,s\right)
^{c}\right)  . \label{expansion}%
\end{equation}
Next,
\begin{align*}
\det\left(  \mathbf{S}_{t}\right)   &  =\det\left(  \mathbf{S}_{t}\left(
\mathbf{i}_{m},\mathbf{i}_{m}\right)  \right)  +\sum_{s\in T\left(
n,p\right)  ,s\neq\mathbf{i}_{m}}\left(  -1\right)  ^{\left\vert s\right\vert
+\left\vert t\right\vert }\det\left(  \mathbf{S}_{t}\left(  t,s\right)
\right)  \det\left(  \mathbf{S}_{t}\left(  t,s\right)  ^{c}\right) \\
&  =Ct^{\kappa_{\left(  n,m\right)  }}+h\left(  t\right)
\end{align*}
where
\[
Ct^{\kappa_{\left(  n,m\right)  }},h\left(  t\right)  =\sum_{s\in T\left(
n,p\right)  ,s\neq\mathbf{i}_{m}}\left(  -1\right)  ^{\left\vert s\right\vert
+\left\vert t\right\vert }\det\left(  \mathbf{S}_{t}\left(  t,s\right)
\right)  \det\left(  \left(  \mathbf{S}_{t}\left(  t,s\right)  \right)
^{c}\right)  \in\mathbb{Q}\left(  \xi\right)  \left[  t\right]  .
\]
From our previous discussion, the monomial $C\times t^{\frac{2m^{2}%
-2mn-2m+n^{2}+n}{2}}$ appears once in the expansion above. Moreover, since $C$
is never equal to zero, it follows that the determinant of $S\left(
\mathfrak{c}\left(  t\right)  \right)  $ is a non-vanishing polynomial in the
variable $t.$ For the second part, it is clear that the degree of the
polynomial $\det\left(  A_{1}\left(  t\right)  A_{4}\left(  t\right)  \right)
$ satisfies the following condition
\[
\deg\left(  \det A_{1}\left(  t\right)  \det A_{4}\left(  t\right)  \right)
\leq m\left(  n-1\right)  +\left(  n-m\right)  \left(  n-1\right)  =n\left(
n-1\right)  .
\]
Next, since the determinant of any sub-matrix of order $n$ of $A\left(
\mathfrak{c}\left(  t\right)  \right)  $ must be a sum of polynomials of
degree less or equal to $n\left(  n-1\right)  $ then the second part of the
lemma holds.
\end{proof}

\section{Proof of Theorem \ref{main}}

First, observe that $n\left(  n-1\right)  +1=n^{2}-n+1$. Fix $\lambda
\in\digamma_{n}$ where $\digamma_{n}$ is as described in (\ref{Fn}). Since the
determinant of every submatrix of $A\left(  \mathfrak{c}\left(  t\right)
\right)  $ of order $n$ is a polynomial in $\mathbb{Q}\left(  \xi\right)
\left[  t\right]  $ of degree bounded above by $n\left(  n-1\right)  ,$ it is
clear that the determinant of every submatrix of $A\left(  \mathfrak{c}\left(
\lambda\right)  \right)  $ of order $n$ cannot be equal to zero. As such, it
immediately follows that $\Lambda\left[  1,\lambda,\cdots,\lambda
^{n-1}\right]  ^{T}$ is a set of cardinality $2n$ in $\mathbb{C}^{n}$ which is
a full spark frame. Consequently, since $\Gamma=\mathbf{F}^{-1}\Lambda
\mathbf{F}\Leftrightarrow\Gamma\mathbf{F}^{-1}=\mathbf{F}^{-1}\Lambda$ it
follows that
\[
\mathbf{F}^{-1}\left(  \Lambda\left[  1,\lambda,\cdots,\lambda^{n-1}\right]
^{T}\right)  =\Gamma\left(  \mathbf{F}^{-1}\left[  1,\lambda,\cdots
,\lambda^{n-1}\right]  ^{T}\right)
\]
is a set of cardinality $2n$ in $\mathbb{C}^{n}$ enjoying the Haar property.

\end{document}